\documentclass[a4paper,12pt]{amsart}
\usepackage{geometry}                		
\usepackage{graphicx}				
\usepackage{amssymb}

\usepackage{latexsym,exscale,enumerate,amsfonts,amssymb,mathtools}
\usepackage{amsmath,amsthm,amsfonts,amssymb,amscd,stmaryrd,textcomp}
\usepackage[normalem]{ulem}
\usepackage{young}
\usepackage{thmtools,xcolor}
\usepackage{easybmat}

\definecolor{colormy}{rgb}{0.8,0.05,0.05}
\definecolor{mycolor}{rgb}{0.25,0.99,0.25}

\usepackage[all]{xy}
\SelectTips{cm}{}

\usepackage{tikz}
\usetikzlibrary{decorations.markings}
\usetikzlibrary{decorations.pathreplacing}
\usetikzlibrary{arrows,shapes,positioning}
\tikzstyle directed=[postaction={decorate,decoration={markings,
    mark=at position #1 with {\arrow{>}}}}]
\tikzstyle rdirected=[postaction={decorate,decoration={markings,
    mark=at position #1 with {\arrow{<}}}}]

\usepackage{hyperref}

\newcommand{\Hom}{\mathrm{Hom}}

\newcommand{\Ext}{\mathrm{Ext}}

\newcommand{\Char}{\mathrm{ch}}

\def\C{{\mathcal C}}

\def\Z{{\mathbb Z}}

\def\K{{\mathcal K}}
\def\ST{{\mathcal{ST}}}
\def\R{{\mathcal R}}

\def\pr{\mathrm{pr}}
\def\Ind{\mathrm{Ind}}
\def\B{\mathcal {B}}
\def\F{\mathcal {F}}
\theoremstyle{definition}
\newtheorem{thm}{Theorem}[section]
\newtheorem{cor}[thm]{Corollary}

\newtheorem{prop}[thm]{Proposition}

\theoremstyle{definition}

\newtheorem{remarkcounter}{Remark}

\numberwithin{equation}{section}

\title{The Steinberg linkage class for a reductive algebraic group}

\author{Henning Haahr Andersen}

\email{h.haahr.andersen@gmail.com}

\date{}							

\begin{document}

\begin{abstract}
Let $G$ be a reductive algebraic group over a field of positive characteristic and denote by $\C(G)$ the category of rational $G$-modules. In this note we investigate the subcategory of $\C(G)$ consisting of those modules whose composition factors all have highest weights linked to the Steinberg weight. This subcategory is denoted $\ST$ and called the Steinberg component. We give an explicit equivalence between $\ST$ and $\C(G)$ and we derive some consequences. In particular, our result allows us to relate the Frobenius contracting functor to the projection functor from $\C(G)$ onto $\ST$.

\end{abstract}

\maketitle



\section{Introduction}

Denote by $k$ an algebraically closed field of characteristic $p > 0$ and let $G$ be a reductive algebraic group over $k$. Then the category $\C(G)$ of rational representations of $G $ splits into components associated to the linkage classes of dominant weights. The Steinberg component $\ST$ associated to the linkage class for the Steinberg weight plays a key role in the representation theory of $G$ and the aim of this note is to investigate this special component. We prove that there is an equivalence of categories (explicitly given in both directions) between $\ST$ and the category $\C(G)$ itself. Moreover, we demonstrate that this equivalence carries the important classes of simple modules, (co)standard modules, indecomposable tilting modules, and injective modules in $\C(G)$ into the corresponding classes in this subcategory.

The above equivalence of categories gives of course an isomorphism between the corresponding Grothendieck groups. The classes of the standard (or Weyl) modules in $\C(G)$ and those in $\ST$ constitute bases in the corresponding Grothendieck groups and we make explicit how the classes in the Grothendieck groups of a given module and its counterpart match up. In particular, when a module has a standard or costandard filtration then our equivalence gives rise to equalities among the number of occurrences of a standard or costandard module in the two equivalent categories.

Among the applications we point out the relation between our equivalence functor and the Frobenius contracting functor studied by Gros and Kaneda in \cite{GK}.

\bigskip

\noindent \textbf{Acknowledgements:} I thank M. Gros, J. E. Humphreys and M. Kaneda for useful comments and corrections.

\section{The category of rational modules}

In this section we introduce notation and and recall along the way some of the basic facts on representations of $G$ that we shall need. For details we refer to \cite{RAG}.

\subsection{Basic Notation}Let $T$ be a maximal torus in $G$ and denote by $X =X(T)$ its character group. In the root system $R \subset X$ for $(G,T)$ we choose a set of positive roots $R^+$ and denote by $X^+ \subset X$  the corresponding cone of dominant characters. Then $R^+$ defines an ordering $\leq$ on $X$. It also determines uniquely a Borel subgroup $B$ whose roots are the set of negative roots $-R^+$.

Denote by $S$ the set of simple roots in $R^+$. Then we define the set of restricted characters $X_1 \subset X$ by 
$X_1 = \{\lambda \in X| 0 \leq \langle \lambda, \alpha^{\vee} \rangle < p \text { for all } \alpha \in S\}$.

We set $\C(G)$ equal to the category of rational $G$-modules. It contains all finite dimensional $G$-modules, and if $M \in \C(G)$ then for each vector $m \in M$ the orbit $Gm$ spans a finite dimensional submodule of $M$.

If $K$ is a closed subgroup of $G$ we write similarly $\C(K)$ for the category of rational $K$-modules. The elements of $X$ are the $1$-dimensional modules in $\C(T)$ and they all extend uniquely to $B$. So if $\lambda \in X$ we shall consider it as an object of $\C(T)$ or of $\C(B)$ as the case may be.

The categories $\C(K)$ all have enough injectives. In particular, all objects of $\C(T)$ are themselves injective. So if $M \in \C(T)$ then $M$ decomposes into a sum of $1$-dimensional modules, i.e. elements of $X$. So we can write
 $$ M = \bigoplus_{\lambda \in X} M_{\lambda} $$
where $M_\lambda = \{m \in M | t m = \lambda(t) m \text{ for all } t \in T\}$. As usual we say that $\lambda \in X $ is a weight of $M$ if $M_\lambda \neq 0$ and we call $M_\lambda$ the $\lambda$-weight space in $M$. When $M$ is finite dimensional we set $\Char M = \sum_{\lambda \in X} \dim M_\lambda e^\lambda$ where $e^\lambda$ is the bases element of $\Z[X]$ corresponding to $\lambda $. We call this the (formal) character of $M$.

\subsection{Induction, costandard and simple 
modules}
The induction functor $\Ind_B^G : \C(B) \rightarrow \C(G)$ is a left exact functor which takes finite dimensional $B$-modules into finite dimensional $G$-modules. We write $H^{i}$ for the $i$-th right derived functor of $\Ind_B^G$. Then also each $H^{i}(E)$ is finite dimensional whenever $E \in \C(B)$ is finite dimensional. In particular, $H^0 = \Ind_B^G$. 

If $\lambda \in X$ then its is wellknown that $H^0(\lambda) \neq 0$ if and only if $\lambda  \in X^+$. In the following we often write $\nabla(\lambda)$ instead of $H^0(\lambda)$ and (when $\lambda \in X^+$) we call this the costandard module in $\C(G)$ with highest weight $\lambda$. The socle of $\nabla (\lambda)$ is simple. We denote it $L(\lambda)$ and the family $\{L(\lambda)\}_{\lambda \in X^+}$ constitutes up to isomorphisms the set of simple modules in $\C(G)$. Then for all $\lambda \in X^+$ 
$$ L(\lambda)_\lambda = \nabla(\lambda)_\lambda =k \text { and } \nabla(\lambda)_\mu \neq 0 \text { implies } \mu \leq \lambda. $$

\subsection{Actions of the Weyl group}

The Weyl group $W = N_G(T)/T$ acts naturally on $X$:  $\lambda \mapsto w(\lambda), \lambda \in X, w \in W$. If $M \in C(G)$ then this action of $W$ permutes the weights of $M$. More precisely, we have $\dim M_\lambda = \dim M_{w(\lambda)}$ for all $\lambda \in X, w \in W$. For $M$ finite dimensional this means that $\Char M \in \Z[X]^W$.

We denote by $\ell$ the length function on $W$ and write $w_0$ for the element in $W$ of maximal length.

In addition to the above action of $W$ on $X$ we shall also consider the socalled "dot-action" given by: $w \cdot \lambda = w(\lambda + \rho) - \rho, w \in W, \lambda \in X$. Here $\rho$ is half the sum of the positive roots.

\subsection{Duality and standard modules}
On $\C(G)$ we consider the duality $ D: \C(G] \rightarrow \C(G)$ given by $DM = \bigoplus_{\lambda \in X} M_\lambda^*$ where on the linear dual module $M^*$ we take the contragredient action composed with the Chevalley automorphism on $G $, cf. \cite{RAG}. Then $D$ preserves characters of finite dimensional modules. In particular, this means that $DL(\lambda) \simeq L(\lambda)$ for all $\lambda \in X^+$.

By Serre duality we have for each finite dimensional $B$-module $E$ a (functorial) isomorphism of $G$-modules $H^{i}(E)^* \simeq H^{N-i}(E^* \otimes (-2\rho))$. Here we take the contragredient actions on the dual modules and $N = \dim G/B = | R^+| = \ell(w_0)$.

We define then $\Delta(\lambda) = H^N(w_0 \cdot \lambda)$. Then $\Delta(\lambda) = 0$ unless $\lambda \in X^+$  and we have $\Delta(\lambda) = D\nabla(\lambda)$. Note that when $\lambda \in X^+$ the head of $\Delta(\lambda)$ is $L(\lambda)$.  
We call $\Delta(\lambda)$ the Weyl (or standard) module for $G$ with highest weight $\lambda$.

\subsection{Filtrations and tilting modules}

A module $M \in \C(G)$ is said to have a $\Delta$- (or a standard, or a Weyl) filtration if it has submodules $M^{i}$ with 
$$ 0=M^0 \subset M^1 \subset \dots \subset M, M^{i+1}/M^{i} \simeq \Delta(\lambda_i) \text { for some  } \lambda_i \in X^+  \text { and } M = \bigcup_i M^{i}.$$
We define $\nabla$ (or costandard, or good,  or dual Weyl) filtrations similarly.

A module $M \in \C(G)$ is called tilting if it has both a $\Delta$- and a $\nabla$-filtration. The first example of a tilting module is the trivial module $k = \Delta(0) = \nabla(0)$. It turns out that the subcategory of $\C(G)$ consisting of the tilting modules is a rich and very interesting subcategory. For each $\lambda \in X^+$ there is a unique (up to isomorphisms) indecomposable tilting module $T(\lambda)$ with highest weight $\lambda)$ (i.e. $T(\lambda)_\lambda = k$ and if $T(\lambda)_\mu \neq 0$ then $\mu \leq \lambda$). The Weyl module $\Delta(\lambda)$ is a submodule of $T(\lambda)$ while the dual Weyl module $\nabla(\lambda)$ is a quotient. The composite of the inclusion $\Delta(\lambda) \to T(\lambda)$ and the quotient map $T(\lambda) \to \nabla(\lambda)$ is up to a constant the unique non-zero homomorphism $\Delta(\lambda) \to \nabla(\lambda)$ (mapping $\Delta(\lambda)$ onto $L(\lambda) \subset \nabla(\lambda)$).

\subsection{The Grothendieck group}\label{Grothendieck}
We denote by $\K(G)$ the Grothendieck group of the {\it finite dimensional} modules in $\C(G)$. This is the abelian group generated by the classes $[M]$ of all finite dimensional modules in $\C(G)$ with relations $[M] = [M_1] + [M_2]$ for all short exact sequences $ 0 \to M_1 \to M \to M_2 \to 0$. 

Note that $\K(G)$ is free over $\Z$ with basis $([L(\lambda)])_{\lambda \in X^+}$. Then for each finite dimensional module $M \in \C(G)$ we can write
\begin{equation}
 [M] = \sum_{\lambda \in X^+} [M:L(\lambda)] [L(\lambda)] 
\end{equation}
for unique non-negative integers $[M:L(\lambda)]$. Then $[M:L(\lambda)]$ is the composition factor multiplicity of $L(\lambda)$ in $M$. Note that these multiplicities are also determined by the character of $M$. In particular,
\begin{equation} \Char M = \Char N \text { iff } [M:L(\lambda)] = [N:L(\lambda)] \text { for all } \lambda \in X^+\text { iff } [M] = [N].
\end{equation}
So in particular, we have $[\Delta(\lambda)] = [\nabla(\lambda)]$ for all $\lambda \in X^+$. 

As $\lambda$ is the unique highest weight of both $\Delta(\lambda)$ and $T(\lambda)$ we see that both of the two families $([\Delta(\lambda)])_{\lambda \in X^+}$ and $([T(\lambda)])_{\lambda \in X^+}$ are also bases of $\K(G)$. If $M \in \C(G)$ we let the integers $(M:\Delta(\lambda))$, respectively $(M:T(\lambda))$ be determined by the equation in $\K(G)$ 
\begin{equation}
 [M] = \sum _{\lambda \in X^+} (M:\Delta(\lambda)) [\Delta (\lambda)], 
\end{equation}
respectively,
\begin{equation}
 [M] = \sum _{\lambda \in X^+} (M:T(\lambda)) [T (\lambda)]. 
\end{equation}

Note that whereas the numbers  $[M: L(\lambda)]$ all are non-negative the numbers $ (M:\Delta(\lambda)) $ and $  (M:T(\lambda))$ may well be negative. However, if $M$ has either a $\Delta$- or a $\nabla$-filtration then $(M:\Delta(\lambda))$ equals the number of occurrencies of $\Delta (\lambda )$ or $\nabla(\lambda)$ in this filtration, see part 2 of Remark \ref{firstremark} below. Also if $M$ is tilting then $(M:T(\lambda))$ counts the number of occurrencies of the indecomposable summand $T(\lambda)$ in $M$.

\begin{remarkcounter}\label{firstremark}
\begin{enumerate}
\item Let $\widehat {\K(G)}$ denote the "completion" of $\K(G)$, i.e. the abelian group consisting of all elements
which may be expressed as (possibly infinite) sums of the form $\sum_{\lambda \in X^+} n_{\lambda}[L(\lambda)]$ with $n_\lambda \in \Z$. Then all those elements $M$  in $\C(G)$ which have finite composition factor multiplicities give welldefined classes $[M] \in \widehat {\K(G)}$. This allows us to give sense to the above numbers $[M:L(\lambda)], (M:\Delta(\lambda))$
and $(M:T(\lambda))$ for all such $M \in \C(G)$. In the following we shall often tacitly pass from $\K(G)$ to $\widehat {\K(G)}$ whenever relevant. 
\item  The fundamental vanishing theorem
\begin{equation}
 \Ext_{\C(G)}^j(\Delta(\lambda), \nabla(\mu)) = 0 \text { unless } \lambda = \mu \in X^+ \text { and } j= 0
\end{equation}
implies that $(M:\Delta (\lambda)) = \sum_j (-1)^j \dim \Ext_{\C(G)}^j(\Delta(\lambda), M)$ for all $M \in \C(G)$. It also implies that
$(M:\Delta (\lambda)) = \dim \Hom_{\C(G)}(\Delta(\lambda), M)$ if $M$ has a $\nabla$ filtration, in which case it therefore counts the number of times that $\nabla(\lambda)$ occurs in a $\nabla$ filtration of $M$.

\item To determine the change of bases matrix $((L(\lambda):\Delta(\mu)))_{\lambda, \mu \in X^+}$ is the subject of the famous Lusztig conjecture \cite{Lu80} which is known to hold for very large primes \cite{AJS} and known to fail for a range of small and not so small primes \cite{W}. At least for $p \geq 2h-2$ ($h$ being the Coxeter number for $R$) this matrix can be deduced from a small part of the change of bases matrix $((T(\lambda):\Delta(\mu)))_{\lambda, \mu \in X^+}$. Recent developments determine this matrix in terms of the socalled $p$-canonical bases, see \cite{RW}, \cite{EL}, \cite{AMRW}.

\end{enumerate}

\end{remarkcounter}

\subsection{Frobenius twist and Steinberg's tensor product theorem}

We denote the Frobenius endomorphism on $G$ (as well as on subgroups) by $F$ and its kernel by $G_1$ (respectively $T$, $B$, ...). Then $G_1$ is a normal subgroup scheme of $G$ with $G/G_1 \simeq G$. 

Let $M \in \C(G)$. When we compose the action by $G$ on $M$  by $F$ we get the socalled Frobenius twist of $M$ which we denote $M^{(1)}$. Conversely, if the restriction to $G_1$ on $V \in \C(G)$ is trivial then there exist $M \in \C(G)$ such that $V = M^{(1)}$. In this case we also write $M = V^{(-1)}$. 

Let $\lambda \in X^+$. Then we write $\lambda = \lambda^0 + p \mu$ for unique $\lambda^0 \in X_1$ and $\mu \in X^+$. With this notation  
the Steinberg tensor product theorem says 
 \begin{equation}\label{Steinberg}
L(\lambda) \simeq L(\lambda^0) \otimes L(\mu)^{(1)}.
\end{equation} 

\subsection{Linkage}

Let $\alpha \in S$. Then $s_\alpha \in W$ is the corresponding reflection given by $s_\alpha (\lambda) = \lambda + \langle \lambda, \alpha^\vee \rangle \alpha$, $\lambda \in X$.
 When $n \in \Z$ we denote by $s_{\alpha, n}$ the affine reflection 
$$s_{\alpha, n}(\lambda) = s_\alpha (\lambda) - np\alpha.$$

The affine Weyl group $W_p$ is the group generated by all $s_{\alpha, n}, \alpha \in S, n \in \Z$. Note that in the Bourbaki convention this is the affine Weyl group corresponding to the dual root systen $R^\vee$.

The linkage principle \cite{A80a} says that whenever $L(\lambda)$ and $L(\mu)$ are two composition factors of an indecomposable module $M \in \C(G)$ then $\mu \in W_p \cdot \lambda$. It follows that $\C(G)$ splits into components according to the orbits of $W_p$ in $X$. More precisely, if we set $A = \{\lambda \in X  | 0 < \langle \lambda + \rho, \alpha^\vee \rangle < p \text { for all } \alpha \in R^+\}$, the bottom dominant alcove, then the closure $\bar A =  \{\lambda \in X  | 0 \leq \langle \lambda + \rho, \alpha^\vee \rangle \leq p \text { for all } \alpha \in R^+\}$ is a fundamental domain for the "dot"-action of $W_p$ on $X$.  If $\lambda \in \bar A$ then the component corresponding to $\lambda$ is the subcategory $\B(\lambda) = \{ M \in \C(G) |  [M:L(\mu)] \neq 0 \text { for some } \mu \in X \text { then } \mu \in W_p \cdot \lambda \}.$ In this notation we have $\C(G) = \bigoplus_{\lambda \in \bar A} \B(\lambda)$.

In this note, we are in particular concerned with the Steinberg component  $ \ST = \B((p-1)\rho)$. Note that the $W_p$-orbit through $(p-1)\rho$ equals $(p-1)\rho + p\Z R$ so that
\begin{equation}\ST = \{M\in \C(G) | [M:L(\mu] \neq 0 \text { implies } \mu \equiv (p-1)\rho \;  (\text{mod } p\Z R) \}.
\end{equation}

\section{The equivalence theorem and some consequences}
In this section we shall assume that $G$ is semisimple of adjoint type so that $X = \Z R$. This simplifies our statements and it is easy to extend our results to more general $G$, cf. part 3. in Remark \ref{furtherremarks}.

In addition to the usual multiplication by $n\in \Z$ on $X$ we shall consider the "dot-multiplication" given by $n\cdot \lambda = n(\lambda + \rho) - \rho$. Like the "dot-action" of $W$ on $X$ "dot-multiplication" also fixes $-\rho$. Moreover, "dot-multiplication" by $n$ commutes with the "dot-actions" of $W$ and $W_p$ on $X$. Note also that in this notation we have $W_p\cdot (p-1)\rho = p \cdot X$. Of course this is also the $W_p$ orbit through $-\rho$ but we prefer to write $(p-1)\rho$ in order to emphasize the close connection to the Steinberg weight. Moreover, $(W_p \cdot (p-1)\rho) \cap X^+ = p \cdot X^+$. This means that the simple modules in $\ST$ are $L(p \cdot \lambda)$ with $\lambda \in X^+$. 

\subsection{The equivalence}
Consider now the following two functors
$$ \F:\C(G) \rightarrow \ST
 \text{ given by } \F M = St \otimes M^{(1)}, \; M \in \C(G)$$
and
$$ \F': \ST  \rightarrow \C(G) \text{ given by } \F' N =\Hom_{G_1}(St, N)^{(-1)}), \; N \in \ST.$$
\begin{thm}\label{equiv}

\begin{enumerate} 

\item{} The functor $\F : \C(G) \rightarrow \ST$ is an equivalence of categories with inverse functor $\F'$.
\item $\F$ and $\F'$ are adjoint functors (left and right). They commute with the duality functor $D$.
\item $\F$ takes simples to simples, (dual) Weyl modules to (dual) Weyl modules, indecomposable tilting, respectively injective, modules in $\C(G)$ to indecomposable tilting, respectively injective, modules in $\ST$. In formulas this is
\begin{equation}\F L(\lambda) = L(p \cdot \lambda), \F\Delta(\lambda) = \Delta(p \cdot \lambda), \F\nabla(\lambda) = \nabla(p \cdot \lambda), \F T(\lambda) = T(p \cdot \lambda),\text { and }  \F I(\lambda) = I(p \cdot \lambda)
\end{equation}
for all $ \lambda \in X^+$.
\end{enumerate}

\end{thm}

\begin{proof} 

Let  Let $M \in \C(G)$ and $N \in \ST$. 

1. First  $\F' (St \otimes M^{(1)}) = \Hom_{G_1}(St, St \otimes M^{(1)})^{(-1)} \simeq \Hom_{G_1}(St, St)^{(-1)} \otimes M \simeq M$, hence $\F' \circ \F  \simeq Id_{\C(G)}$. To prove that $\F \circ \F' \simeq Id_{\ST}$ we need to prove that $St \otimes \Hom_{G_1} (St, N) \simeq N$. If $N$ is simple, i.e. $N = L(p \cdot \lambda)$ for some $\lambda \in X^+$, then by Steinberg's tensor product theorem (\ref{Steinberg}) $N \simeq St \otimes L(\lambda)^{(1)}$. So in this case $\Hom_{G_1}(St, N) \simeq L(\lambda)^{(1)}$ and $(\F \circ \F') N \simeq N$. Now the functor $\Hom_{G_1}(St, -)$ is exact ($St$ is projective as a $G_1$-module) and therefore induction on the length of $N$ (as $G$-module) proves the statement in general.
 
2. We have $\Hom_{G} (\F M, N)  = \Hom_G(St \otimes M^{(1)}, N) \simeq \Hom_G(M^{(1)} ,\Hom_{G_1}(St, N)) = \Hom_G(M, \F' N)$, i.e. $\F$ is left adjoint to $\F'$. Right adjointness follows then from 1. Since $D St \simeq St$ we easily see that $D (\F M) \simeq \F (DM)$. As $\F'$ is the inverse of $\F$ we get also $D \circ \F' =\F' \circ D$.

3.
 When it comes to simple modules it is part of Steinberg's tensor product theorem that $\F$ takes $L(\lambda)$ into $L(p \cdot \lambda)$. It is a special case of the Andersen-Haboush theorem (see \cite{A80b} and \cite{Ha})  that $\F \nabla(\lambda) \simeq \nabla
(p \cdot \lambda)$. It then follows that the exact functor $\F$ will take tiltings to tiltings. Moreover, being an equivalence $\F$ also takes indecomposables to indecomposables. Hence $\F T(\lambda) \simeq T(p\cdot \lambda)$. Finally, $\F$ clearly takes injectives to injectives. 

\end{proof}

\subsection{First consequences}

As immediate consequences of this theorem we get the following two corollaries.

\begin{cor}
Let $M \in \C(G)$. Then we have 
\begin{enumerate} 
\item
$M$ is semisimple iff $St \otimes M^{(1)}$ is.
\item
$M$ is indecomposable iff $St \otimes M^{(1)}$ is.
\item
$M$ has a Weyl (respectively dual Weyl) filtration iff $St \otimes M^{(1)}$ does.
\item
$M$ is tilting iff $St \otimes M^{(1)}$ is.
\item
$M$ is injective iff $St \otimes M^{(1)}$ is injective.
\end{enumerate}
\end{cor}

\begin{cor}
The entries in the change of bases matrices for the three bases for $\K(G)$ described in Section \ref{Grothendieck} and the corresponding ones in $\K(\ST)$ match up as follows
\begin{equation}
[\Delta(\lambda):L(\mu)] = [\Delta(p \cdot\lambda):L(p\cdot \mu)], (T(\lambda):\Delta(\mu)) = (T(p \cdot\lambda):\Delta(p\cdot \mu))
\end{equation}
for all $\lambda, \mu \in X^+$.
\end{cor}

\subsection{Projection onto the Steinberg component}

Let $M \in \C(G)$. Then we have $M = \oplus_{\eta \in \bar A} M(\eta)$ where $M(\eta) $ is the largest submodule of $M$ belonging to $\B(\eta)$. We also write $M(\eta) = \pr_\eta (M)$ where $\pr_\eta$  denotes the projection functor $\C(G) \to \B(\eta)$. Note that $\pr_\eta$ is adjoint to the inclusion functor from $\B(\eta)$ into $\C(G)$. 

The projection functor onto the $\ST$  is also denoted $\pr_\ST$. This functor is given by

\begin{prop}
Let $M \in \C(G)$. Then $\pr_\ST (M) = St \otimes \Hom_{G_1}(St, M)$
\end{prop}

\begin{proof}
Since  $\pr_\ST(M) \in \ST$ it follows from Theorem \ref{equiv}  that $\pr_\ST (M) = St \otimes N^{(1)}$ for some module $N \in \C(G)$. But then $N^{(1)} =\Hom_{G_1} (St, St \otimes N^{(1)}) = \Hom_{G_1} (St, \pr_ST (M)) = \Hom_{G_1} (St, M).$ Here the last equality comes from the fact that $\Hom_{G_1}(St, V) = 0$ if $V$ belongs to a component in $\C(G)$ different from $\ST$. 
\end{proof}

\subsection{Relations with induction and with Frobenius contraction}

The following result gives the relation between the above equivalence of categories and the derived 
functors of induction from $B$ to $G$. We denote by $\F_B$ the endofunctor on $\C(\B)$ given by $\F_B E = E^{(1)} \otimes (p-1)\rho, \; E \in \C(\B)$.

\begin{thm}
We have $\F \circ R^j \Ind_B^G =R^j \Ind_B^G \circ \F_B$ and $R^j \Ind_B^G =  \F' \circ R^j \Ind_B^G \circ \F_B$ for all $j$. In particular, if $\lambda \in X$ then we have $St \otimes H^j(\lambda)^{(1)} = H^j(p \cdot \lambda)$ and  $H^j (\lambda) = \Hom_{G_1} (St, H^j( p \cdot \lambda))^{(-1)}$.
\end{thm}

\begin{proof}
The first identity comes from \cite{A80b}. The second one follows from the first since $\F'$ is the inverse of $\F$. The special case results from applying these two identities  to the $B$-module $\lambda$.

\end{proof}

Recall that Gros and Kaneda\cite{GK} have introduced the Frobenius contracting functor $\phi$ on $\C(G)$. This is a right adjoint of the Frobenius twisting functor  $M \mapsto M^{(1)}$ on $\C(G)$ composed with tensoring twice with $St$. As a $T$-module $\phi$ is determined by $(\phi M)_\lambda = M_{p\lambda}$ for all $\lambda \in X$.

\begin{prop} \label{contract}
The Frobenius contracting functor $\phi$ on $\C(G)$ is the composite $\F' \circ \pr_\ST \circ [St \otimes -]$.
\end{prop}

\begin{proof} By Theorem 2.1 in \cite{GK} (attributed to S. Donkin) we have $\phi M = \Hom_{G_1}(St, St \otimes M)^{(-1)}, \; M \in \C(G)$. As observed above we have $\Hom_{G_1}(St, -) = \Hom_{G_1}(St, \pr_\ST \circ -)$ and the proposition therefore follows from Theorem \ref{equiv}.
\end{proof}

The following corollary is Theorem 3.1 in \cite{GK}

\begin{cor}
If $ M \in \C(G)$  has a Weyl filtration, respectively a dual Weyl filtration, so does $\phi M$.
\end{cor}

\begin{proof} 
If $M$ has a filtration of one of the types in the corollary then so does $St \otimes M$ (because $St$ is a selfdual Weyl module). This property is then inherited by the direct summand $\pr_\ST (St \otimes M)$. Finally $\F'$ preserves this property, see Theorem \ref{equiv}. Hence the corollary is an immediate consequence of Proposition \ref{contract}.
\end{proof}

\subsection{Further remarks}

\begin{remarkcounter}\label{furtherremarks}
\begin{enumerate}
\item 
Let $\F^r = \F \circ \F \circ \dots \circ \F$ denote the composite of $\F$ by itself $r$ times, $r \in \Z_{\geq 0}$. Also write $St_r = L((p^r-1)\rho)$. Then by the Steinberg tensor product theorem we see that $\F^r M = St_r \otimes M^{(r)}$, $M \in \C(G)$. Here $^{(r)}$ denotes twist by the $r$-th Frobenius $F^r$. Denote by $\ST_r$ the subcategory of $\C(G)$ consisting of all modules whose composition factors have the form $L(p^r \cdot \lambda)$ with $\lambda \in X^+$. Then we have a chain of subcategories
$$ \dots \subset  \ST_r \subset \ST_{r-1} \subset \dots \subset \ST_1 = \ST \subset \ST_0 = \C(G).$$
We have that $\ST$ is a summand in $C(G)$. More precisely, $C(G) = \ST \oplus \R$ where $\R$ is the subcategory consisting of those modules, which have composition factors with highest weights in $X^+\setminus p\cdot X^+$.
The restriction of $\F$ to $\ST$ gives an equivalence between $\ST$ and $\ST_2$ (with inverse the restriction of $\F'$ to $\ST_2$). Iterating this we see that in the above chain all the subcategories are equivalent to $\C(G)$ and for each $r$ we have $\C(G) = \ST_r \oplus \R_r$ where $\R_r$ consists of all modules whose composition factors are in $X^+\setminus p^r \cdot X^+$.

\item
Consider the quantum group $U_q$ corresponding to $G$ at a complex root of unity $q$. Let $\ell$ denote the order of $q$ and assume $\ell$ is odd and if $R$ contains a subsystem of type $G_2$ then also that $\ell$ is not divisible by $3$. Then we have a quantum analoque (using obvious notation)  $\F_q : \C(G_{\mathbb C}) \to \ST_q$ of the above equivalence (in the obvious notation) given by $\F_q (M) = St_q \otimes M^{[q]}$. Here $^{[q]}$ denotes the $U_q$-module obtained by precomposing the representation of $G_{\mathbb C}$ on $M$ by the quantum Frobenius homomorphism, see \cite{Lu94}.

Recall that the category $\C(G_{\mathbb C}$ is semisimple (each dominant weight $\lambda$ determines a component whose only simple module is $ L_{\mathbb C}(\lambda) $). Hence $\ST_q$ is also semisimple and we have for each $\lambda \in X^+$
$$ L_q(\ell \cdot \lambda) =  \Delta_q(\ell \cdot \lambda) = \nabla_q(\ell \cdot \lambda) = T_q(\ell \cdot \lambda).$$

Note that in this case we have no "higher" Steinberg subcategories (we cannot iterate the quantum Frobenius nor can we take powers pf $\F_q$). However, if instead of taking a complex root of unity  (or more generally a root of unity in a characteristic $0$ field) we let $q$ be a root of unity in a characteristic $p$ field  then we do have a similar sequence as in 1. above. This time the first term is $\ST_q$ (involving the quantum Frobenius) and the higher Steinberg subcategories are obtained via the Frobenius endomorphism on $G$. We leave the details to the reader. 
\item
By a special point in $X$ we understand a weight $\lambda \in X$ which satisfies $\langle \lambda + \rho, \alpha^\vee \rangle \in p\Z$ for all $\alpha \in R$. Our assumption that $G$ is adjoint implies that all special points are in the same orbit under the action by $W_p$ (in fact already under the action of the subgroup $ p\Z R = p X \subset W_p$). If we relax the assumption that $G $ is adjoint and only assume that $G$ is semisimple then there may be several different orbits of special points. If $\nu \in \bar A$ is a special point then we have a component $\B(\nu)$ in $\C(G)$ consisting of all modules whose compositions factors have highest weights coming from the orbit of $\nu$. We can then define "higher" components $\B_r(\nu)$ by taking the images under restriction to $\B(\nu)$ of $\F^r$. In this way all the results in this section generalizes to non-adjoint groups.  

\end{enumerate}
\end{remarkcounter}

\section{Formulae in the Grothendieck groups}

In \cite{A89} we proved some character formulae. Here we start by giving the corresponding formulae in $\K(G)$.

\begin{prop}
Let $M \in \C(G)$ be finite dimensional. Then we have
\begin{equation}
[M] = \sum_{\lambda \in X^+}(\sum_{w \in W}  (-1)^{\ell(w)}  \dim M_{w \cdot \lambda}) [\Delta(\lambda)] =  \sum_{\lambda \in X^+}(\sum_{w \in W}(-1)^{\ell(w)}  \dim M_{\lambda -w \cdot 0}) [\Delta(\lambda)]
\end{equation}

More generally, we have for all weights $\mu \in X^+$
\begin{equation} 
[\Delta(\mu) \otimes M]  =  \sum_{\lambda \in X^+}(\sum_{w \in W}(-1)^{\ell(w)}  \dim M_{w\cdot \lambda - \mu}) [\Delta(\lambda)].
\end{equation}

\end{prop}

\begin{proof}
In the first formula the second equality comes from the fact that $\dim M_\mu = \dim M_{x(\mu)}$ for all $\mu \in X$ and $x \in W$.  That the last term equals $[M]$ is the content of Corollary 3.6 in \cite{A89}. Note that to check this it is by additivity enough to verify that if $\mu \in X^+$ then $\sum_{w \in W}(-1)^{\ell(w)}  \dim \Delta(\mu)_{\lambda -w \cdot 0} = \delta_{\mu, \lambda}$. This in turn is a special case of Proposition 4.3 in \cite{A89}.

To check the second identity it is again by additivity enough to check it for $M = \Delta (\nu), \; \nu \in X^+$. In this  case the formula is equivalent to the equation found in the last Remark in \cite{A89}.

\end{proof}

We are especially interested in the above identities when the weights in question belong to the $W_p$-orbit of the Steinberg weight $(p-1)\rho$. Here we have 

\begin{cor}
Let $M \in \C(G)$ be finite dimensional. Then we have
$$ (St \otimes M : \Delta(p\cdot \lambda)) = \sum_{w \in W} (-1)^{\ell(w)} \dim M_{pw\cdot \lambda}$$
for all $\lambda \in X^+$.
\end{cor}

\begin{proof}
This comes from the second part of Proposition 4.1 by taking $\mu = (p-1)\rho$ and noting that $w\cdot p\cdot \lambda -(p-1)\rho = p w\cdot \lambda$.

\end{proof}

\begin{remarkcounter}
If in the above corollary we take $M = N^{(1)}$ for some $N \in \C(G)$ then we get $(St \otimes N^{(1)} : \Delta(p\cdot \lambda)) = \sum_{w \in W} (-1)^{\ell(w)}  \dim N_{w\cdot \lambda}$. By Proposition 4.1 this equals $(N:\Delta(\lambda))$ consistent with our equivalence from Theorem \ref{equiv}.
\end{remarkcounter}

\begin{prop} \label{hom-st}
Let $M \in \C(G)$. Then in $\K(G)$ we have $[\Hom_{G_1}(St, M)^{(-1)}] = \sum_{\lambda \in X^+} (M:\Delta(p\cdot \lambda)) [\Delta(\lambda)].$
\end{prop}

\begin{proof}
As $\Hom_{G_1} (St, -)$ is exact, the left hand side is additive on exact sequences. So is the right hand side and hence it is enough to check the identity for $M = \nabla(\mu), \; \mu \in X^+$. Now if $\mu \notin p \cdot X^+$ then both sides are $0$. On the other hand if $\mu = p \cdot \lambda$ for some $\lambda \in X^+$ then $\nabla(\mu) = St \otimes \nabla(\lambda)^{(1)}$ and we get $\Hom_{G_1} (St, \nabla(\mu))^{(-1)} = \nabla(\lambda)$ proving the formula.
\end{proof}

This proposition allows us to obtain a formula for the class of the Frobenius contraction of an arbitrary module in $\C(G)$, cf \cite{GK}.

\begin{cor}
Let $M \in \C(G)$. Then in $\K(G)$ we have
$$  [\phi M] = \sum_{\lambda \in X^+} (St \otimes M : \Delta(p\cdot \lambda)) [\Delta(\lambda)].$$
\end{cor}

\begin{proof}
We note (as we did a couple of times in Section 3) that $(St \otimes M : \Delta(p\cdot \lambda)) = (\pr_\ST (St \otimes M) : \Delta(p\cdot \lambda))$. Then the corollary comes by combining Propositions \ref{contract} and \ref{hom-st}.

\end{proof}

\begin{remarkcounter}
In the case where $M = N^{(1)}$ the corollary says (when applying our equivalence theorem)  that $[\phi N^{(1)}] = [N]$. This is consistent with the fact that $\phi$ operates as untwisting when applied to a twisted module, cf. \cite{GK}.
\end{remarkcounter}

\vskip 1 cm 
\end{document}